\documentclass[letterpaper, conference]{ieeeconf}
\IEEEoverridecommandlockouts
\usepackage{masih-ieee}
\usepackage{titlesec}

\newcommand{\aug}{\text{aug}}
\newcommand{\Taug}{\Tc^{\aug}}
\newcommand{\Kaug}{\Kc^{\aug}}
\newcommand{\Faug}{\Fc^{\aug}}
\newcommand{\bestapprox}[1]{\mathrel{\underset{#1}{\overset{\scriptscriptstyle\mathrm{best}}{\approx}}}}
\newcommand{\RRMSE}{\operatorname{RRMSE}}

\renewcommand{\baselinestretch}{.972}

\begin{document}

\title{Control Forward-Backward Consistency: Quantifying the Accuracy of Koopman Control Family Models}
\author{Masih Haseli \qquad Jorge Cort\'es \qquad Joel W. Burdick
		\thanks{M. Haseli and J. W. Burdick  are with the Division of Engineering and Applied Science, California Institute of Technology, Pasadena, CA 91125, USA, {\tt\small\{mhaseli, jburdick\}@caltech.edu}. J. Cort\'es is with the Department of Mechanical and Aerospace Engineering, University of California, San Diego, CA 92093, USA, {\tt\small cortes@ucsd.edu}.}
        \thanks{The work of M. Haseli and J. Burdick was supported by the Technology Innovation Institute of Abu Dhabi and DARPA under the LINC program. The work of J. Cort\'es was supported by ONR Award N00014-23-1-2353.}
}
\maketitle

\begin{abstract}
This paper extends the forward-backward consistency index, originally introduced in Koopman modeling of systems without input, to the setting of control systems, providing a closed-form computable measure of accuracy for data-driven models associated with the Koopman Control Family (KCF). Building on a forward-backward regression perspective, we introduce the control forward-backward consistency matrix and demonstrate that it possesses several favorable properties.
Our main result establishes that the relative root-mean-square error of KCF function predictors is strictly bounded by the square root of the control consistency index, defined as the maximum eigenvalue of the consistency matrix. This provides a sharp, closed-form computable error bound for finite-dimensional KCF models. 
We further specialize this bound to the widely used lifted linear and bilinear models.
We also discuss how the control consistency index can be incorporated into optimization-based modeling  and illustrate the methodology via simulations. 
\end{abstract}

\vskip -0.1 true in
\section{Introduction}
Koopman-based methods have become 
popular in controls and robotics, with many applications relying on finite-dimensional Koopman-inspired models. Since the quality of the resulting closed-loop behavior directly depends on the accuracy of the underlying model, measuring this accuracy is of utmost importance. 
A proper accuracy measure is essential not only for verification but also for reliable construction of data-driven models (e.g., as a loss function). 
This paper develops such a measure for control systems.

\emph{Literature Review:}
The Koopman operator associates a dynamical system with a linear operator on a function space~\cite{BOK:31}, capturing how functions evolve under the dynamics. 
Although the original Koopman operator definition does not account for system inputs, there are numerous successful use cases of Koopman-based methods in controls, e.g. optimal~\cite{AS-AM-DE:18}, robust~\cite{RS-MS-KW-JB-FA:25}, and safety-critical control~\cite{CF-YC-ADA-JWB:20}.

Many prior works do not rely on a formal extension of the Koopman operator to control systems; instead, they rely on ``lifted'' models inspired by the finite-dimensional form of Koopman operators for systems without input~\cite[Section~1.4]{AM-YS-IM:20}. Lifted linear forms~\cite{JLP-SLB-JNK:16} are widely used for their compatibility with linear control methods. However, such forms do not capture input nonlinearities or the input-state coupling. 
The work~\cite{XS-MH-JC-YZ:26-tac} provides necessary and sufficient conditions on the system dynamics to admit a finite-dimensional lifted linear form; while \cite{DU-KD:25} studies and tackles the limitations of lifted linear models via oblique projections.
Bilinear forms are another widely used Koopman-based surrogate model, often associated with control-affine systems. 
The work~\cite{SP-SEO-CWR:20} studies the Koopman generators associated with control-affine systems and bilinear models while~\cite{DG-DAP:21} studies a similar problem through the lens of geometric control.  We refer the reader to~\cite{RS-KW-IM-JB-MS-FA:26} for a review of Koopman-based control and error bounds, especially for lifted linear and bilinear forms.  
General non-affine control systems can be handled in a Koopman framework by setting the control input to be a constant, whereupon the system admits a traditional Koopman operator.
Thus control inputs can be handled by switching between the Koopman operators for finitely many constant-input systems. This idea has been used numerous times, for purposes such as control of PDEs~\cite{SP-SK:19} and monotone systems~\cite{AS-AM-DE:18}, and sampling-based control design~\cite{JH-KC:24-lcss}.
 
While the bulk of the literature studied Koopman-based surrogate models,  rigorous extensions of Koopman theory to general (not necessarily control affine) nonlinear  control systems have received less attention.  \cite{MK-IM-automatica:18} builds a rigorous 
extension of Koopman operator theory to general discrete-time control systems by considering all possible behaviors arising from infinite-input sequences. Alternatively, \cite{MH-JC:26-auto} introduces the {\em Koopman Control Family} (KCF) formalism, using all possible Koopman-operators associated with all constant-input systems. Later, it was shown in~\cite{MH-IM-JC:25-tac} that under certain conditions on the function spaces, the infinite-input sequence paradigm~\cite{MK-IM-automatica:18} is in fact equivalent to KCF~\cite{MH-JC:26-auto}, although the theoretical structures look very different. Alternatively,~\cite{ML:25} provides a different extension of the Koopman operator to control systems based on product spaces.

This paper studies the accuracy of finite-dimensional forms associated with the KCF framework in data-driven settings. We extend the forward-backward consistency index from~\cite{MH-JC:23-csl}, originally defined for systems without input, to control systems.  The notion of forward and 
backward dynamics has been used numerous times in the literature, including to deal with noise~\cite{STMD-MSH-MOW-CWR:16,LL-SD-JRF:26}, forecasting~\cite{OA-NBE-VL-MM:20}, and identifying close to Koopman-invariant subspaces~\cite{MH-JC:25-access}, among many other applications.

\emph{Statement of Contributions:}
We extend the \emph{consistency index} used for Koopman-based modeling of systems without input~\cite{MH-JC:23-csl} to control systems, and thereby develop an accuracy measure for input-state separable models associated with the Koopman Control Family (KCF) for data-driven settings. Adopting a forward-backward regression perspective, we introduce the \emph{control forward-backward consistency matrix} (CFBCM). We show that this matrix has several desirable properties: (i) it is similar to a positive-semidefinite matrix; (ii) its eigenvalues are real and belong to the interval $[0,1]$, and (iii) these eigenvalues do not depend on the choice of basis. 
This invariance of the spectrum under the choice of basis is quite meaningful since function predictors in the KCF are also independent of the choice of basis.
Our main result establishes that the relative root-mean-square error of KCF function predictors admits a sharp upper bound determined by the square root of the largest CFBCM eigenvalue (which we define as the \emph{control consistency index}), leading to 
a closed-form error bound. We also show how this bound can be specialized to the widely used lifted linear and bilinear subclasses of input-state separable models. Finally, we show how our results can be used in optimization-based modeling and provide guidelines for improving training efficiency\footnote{We denote the sets of natural, non-negative integer, real, and complex numbers by $\naturals$, $\naturals_0$, $\real$, and $\cplx$.  
For matrix $A \in \cplx^{m \times n}$, we use $A^T$, $A^*$, $\|A\|_F$, $A^\dagger$, $\row(A)$, and $\text{cond}(A)$  to denote its transpose, conjugate transpose, Frobenius norm, pseudoinverse, row space (the vector space spanned by rows), and condition number. If $A$ is square, $\tr(A)$, $A^{-1}$, and $\spec(A)$ denote its trace, inverse, and spectrum. If $\spec(A) \subset \real$, $\lambda_{\max}(A)$ denotes its maximum eigenvalue. If $A$ is positive semidefinite, $A^{\frac{1}{2}}$ denotes its square root. For matrices $A$ and $B$ with the same number of rows, $[A,B]$ denotes the side-by-side concatenation of $A$ and $B$. $0_{m \times n}$ and $I_m$ represent the $m \times n$ zero matrix and $m \times m$ identity matrix respectively (indices are omitted when clear). For $v \in \cplx^n$, $\| v \|$ denotes its $2$-norm. For vector space $\Vc$, $\dim(\Vc)$ denotes its dimension. For set $S \subseteq \Vc$, $\Span(S) \subseteq \Vc$ denotes the vector space spanned by $S$. For function $\Psi: \Xc \to \cplx^n$, $\Span(\Psi)$ denotes the vector space spanned by its elements. Given
 $f: A \to B$ and $g: B \to C$, $g \circ f: A \to C$ denotes their composition. For sets $B \subseteq A$ where $A$ is equipped with a metric $d$, we say $y$ is the best approximation of $x \in A$ on $B$ denoted by $x \bestapprox{B} y$ if $y \in \arg \min_{z \in B} d(x,z)$.}.

\section{A primer on Koopman Control Family}
Here, we provide a few necessary definitions and results from the  Koopman Control Family (KCF) framework following~\cite{MH-JC:26-auto} that we use throughout the paper.
Consider the (not necessarily control-affine) discrete-time dynamical system
\begin{align}\label{eq:control-system}
x^+ = \Tc(x, u), \quad x \in \Xc \subseteq \real^n, \; u \in \Uc \subseteq \real^m.
\end{align}
By setting input $u$ to be a constant, one can use~\eqref{eq:control-system} to create a family of systems without input
\begin{align}\label{eq:constant-input-systems}
\{x^+ = \Tc_{u^*}(x):= \Tc(x, u \equiv u^*)\}_{u^* \in \Uc}.
\end{align}
Given that the family of systems~\eqref{eq:constant-input-systems} do not have inputs, the traditional definition of the Koopman operator can be applied to the family, leading to the definition of the KCF. 
\begin{definition}\longthmtitle{Koopman Control Family~\cite[Definition~4.1]{MH-JC:26-auto}}\label{def:kcf}
Consider $\Fc$, a vector space of complex-valued functions with domain $\Xc$ that is closed under composition with elements of the family~\eqref{eq:constant-input-systems}, i.e., $f \circ \Tc_{u^*} \in \Fc$, for all $\map{f}{\Xc}{\cplx}$ in $\Fc$ and all $u^* \in \Uc$. The Koopman Control Family (KCF) is defined as $\{\Kc_{u^*} : \Fc \to \Fc \}_{u^* \in \Uc}$, where
\begin{align*}
\Kc_{u^*}f = f \circ \Tc_{u^*}, \quad \forall f \in \Fc, \; \forall u^* \in \Uc. 
\eqoprocend
\end{align*}
\end{definition}
\smallskip
The KCF can fully encode the trajectories of~\eqref{eq:control-system} in the evolution of functions in $\Fc$:  given trajectory $\{x_k\}_{k \in \naturals_0}$ from the initial condition $x_0$ and input sequence $\{u_k\}_{k \in \naturals_0}$,
\begin{align}\label{eq:trajectory-encoding}
f(x_k) = \left[ \Kc_{u_0} \cdots \Kc_{u_{k-1}} f \right](x_0), \quad \forall f \in \Fc, \; \forall k \in \naturals.
\end{align}

\subsection{Parameterizing KCF via Augmented Koopman Operator}
Since the KCF may contain uncountably many operators, utilizing it in an efficient manner requires an effective parameterization.  
To this end, we first parameterize the family of constant-input systems~\eqref{eq:constant-input-systems} via the following system that acts on the augmented state space $\Xc \times \Uc$
\begin{align}\label{eq:Taug}
(x, u)^+ = \Taug(x,u) := (\Tc(x, u), u).
\end{align}
Note that the augmented system $\Taug$ in~\eqref{eq:Taug} does not have an input, and $u$ is a part of the state. System~\eqref{eq:Taug} completely captures the family of constant input systems~\cite[Lemma~5.1]{MH-JC:26-auto}.
Given that~\eqref{eq:Taug} is a system without input, we can define a Koopman operator $\Kaug: \Faug \to \Faug$ as
\begin{align}\label{eq:Kaug}
\Kaug g = g \circ \Taug, \quad \forall g \in \Faug,
\end{align}
where $\Faug$ is an appropriate function space closed under composition with $\Taug$. Interestingly, $\Kaug$ can parametrize the members of the KCF given mild conditions on function spaces $\Faug$ and $\Fc$: see ~\cite[Definition~5.4]{MH-JC:26-auto}.  The most important condition, which we use frequently, is $f \circ \Tc \in \Faug$ for all $f \in \Fc$. We refer the reader to the discussion after~\cite[Definition~5.4]{MH-JC:26-auto} for examples for choices of $\Fc$ and $\Faug$.

\subsection{Finite-Dimensional Forms of KCF}
For computational reasons, one naturally seeks finite-dimensional approximations to the KCF. Finite-dimensional models arise from the principle of subspace invariance: if a model’s input and output are to lie in a given finite-dimensional subspace, that subspace must be invariant under \emph{the action of the model}. For dynamical systems without input, finite-dimensional Koopman-based models follow the same principle: their form is dictated by subspaces that are invariant under the Koopman operator, 
leading to the well-known lifted linear model.  The same form is used for non-invariant subspaces  
through \emph{approximations}\footnote{These approximations often involve composing projection operators with the Koopman operator to generate an approximate operator (this is known as compression) that renders the underlying subspace invariant~\cite[Section~1.4]{AM-YS-IM:20}.}. 

To obtain a finite-dimensional form for the KCF using the invariance principle, let $H: \Xc \to \cplx^{n_\Hc}$ be such that its elements form a basis for a KCF's common invariant subspace, denoted by $\Hc \subseteq \Fc$. The finite-dimensional form of the KCF, termed the ``input-state separable'' model, is 
\begin{align}\label{eq:input-state-separable-function-form}
H \circ \Tc  = \Ac \, H,
\end{align}
where $\Ac: \Uc \to \cplx^{n_{\Hc} \times n_{\Hc}}$ is a matrix-valued function.
Evaluating~\eqref{eq:input-state-separable-function-form} on a pair $(x,u) \in \Xc \times \Uc$ gives $H(x^+) = H \circ \Tc (x) = A(u) H(x)$, leading to the following point-wise dynamic predictor\footnote{Since Koopman operators act on functions; Eq.~\eqref{eq:input-state-separable-function-form} is a fundamental notion, while its point-wise evaluation~\eqref{eq:input-state-separable-pointwise} is simply a useful byproduct. Later, for the case of non-invariant subspaces, we study \emph{function} approximations. 
}
\begin{align}\label{eq:input-state-separable-pointwise}
z^+ = \Ac(u) z, \quad z \in \cplx^{n_\Hc},  u \in \Uc.
\end{align}
Given trajectory $\{x_k\}_{k \in \naturals_0}$ generated from the initial condition $x_0$ and input sequence $\{u_k\}_{k \in \naturals_0}$, if we initialize~\eqref{eq:input-state-separable-pointwise} with $z_0= H(x_0)$, we have $z_k = H(x_k)$ for all $k \in \naturals_0$.

\begin{remark}\longthmtitle{Generality of Input-State Separable Models}\label{r:generality}
The input-state separable form (\ref{eq:input-state-separable-pointwise}) is the general finite-dimensional form of the KCF following~\cite[Theorem~4.3]{MH-JC:26-auto}.   
The lifted linear, bilinear, and switched linear models that are widely used in the Koopman literature are special cases of input-state separable models~\cite[Lemmas~4.5--4.6]{MH-JC:26-auto}.
\oprocend
\end{remark}

Eq.~\eqref{eq:input-state-separable-pointwise} is linear in the ``lifted state'' $z$, but is generally nonlinear in $u$. This nonlinearity in the input is unavoidable whenever system~\eqref{eq:control-system} is nonlinear in the input.

Finite-dimensional invariant subspaces containing the desired functions may not be easy 
to identify, or even exist. In such cases, we use an approximate form of~\eqref{eq:input-state-separable-function-form}
\begin{align}\label{eq:input-state-separable-function-form-approximate}
H \circ \Tc  \approx  \Ac \, H. 
\end{align}
To approximate the action of functions on trajectories, one still uses~\eqref{eq:input-state-separable-pointwise} with $\Ac$ in~\eqref{eq:input-state-separable-function-form-approximate}; however,  the action of $H$ on trajectories of~\eqref{eq:control-system} deviates from the trajectories of~\eqref{eq:input-state-separable-pointwise}, and this deviation depends on the quality of approximation in~\eqref{eq:input-state-separable-function-form-approximate}.  It is important to note that~\eqref{eq:input-state-separable-function-form-approximate} provides a predictor not only for the evolution of elements of $H$, but also for all elements of $\Hc = \Span(H)$. Indeed, given $h \in \Hc$ with representation $h = v_h^* H$ for some $v_h \in \cplx^{n_\Hc}$, the approximation of $h \circ \Tc$ induced by~\eqref{eq:input-state-separable-function-form-approximate}, denoted by $\Pf_{h \circ \Tc}$, is given by
\begin{align}\label{eq:KCF-function-predictor}
	h \circ \Tc \approx \Pf_{h \circ \Tc} := v_h^* \Ac H.
\end{align}
The approximations in~\eqref{eq:input-state-separable-function-form-approximate}--\eqref{eq:KCF-function-predictor} are obtained via projections, defined according to the structure of $\Fc$ and $\Faug$.

Next, we describe how $\Kaug$ can be used to construct \emph{optimal} predictors  when $\Faug$ is endowed with an inner product
$\innerprod{\cdot}{\cdot}$ inducing a norm $\| \cdot \|$. To this end, we rely on the following subspaces.

\begin{definition}\longthmtitle{Non-degenerate Normal Subspaces in $\Faug$ and Their Bases~\cite[Definition~7.1]{MH-JC:26-auto}}\label{def:normal-space}
Let the elements of $\Psi: \Xc \times \Uc \to \cplx^{n_{\Psi}}$ be linearly independent. We say $\Psi$ is in \textbf{(non-degenerate)
normal separable form} (or normal form for short) if it can be decomposed as 
\begin{align}\label{eq:normal-basis}
\Psi&=\begin{bmatrix} I_{\Uc}\\ G\end{bmatrix}H
\end{align}
where  $H: \Xc \to \cplx^{n_{\Hc}}$, $G: \Uc \to \cplx^{(n_{\Psi}-n_{\Hc}) \times n_{\Hc}}$ and function $I_{\Uc}: \Uc \to \cplx^{n_{\Hc} \times n_{\Hc}}$ returns the identity matrix,  $I_{\Uc}(u) = I$ for all $u \in \Uc$.  We say that $\Sc \subset \Kaug$ is a \textbf{normal subspace} if its basis can be written in normal form\footnote{ In~\eqref{eq:normal-basis}, $I_{\Uc}$ and $G$ are matrix-valued functions with domain $\Uc$ and $H$ is a vector-valued function with domain $\Xc$. The matrix products are defined in an element-wise sense similar to  numerical matrix and vector products.}. \oprocend
\end{definition}
\smallskip

The next result connects the finite-dimensional approximations of $\Kaug$ on $\Sc = \Span(\Psi) \subseteq \Faug$ to the input-state separable forms on $\Hc = \Span(H) \subseteq \Fc$.

\begin{theorem}\longthmtitle{Optimal Approximation of Input-State Separable Forms from $\Kaug$~\cite[Theorem~8.7]{MH-JC:26-auto}}\label{t:optimal-approx-Kaug}
Given $\Psi$ and $H$ in~\eqref{eq:normal-basis}, let $\Sc = \Span(\Psi) \subseteq \Faug$, $\Hc = \Span(H)$, and let $\Pc_{\Sc}: \Faug \to \Faug$ be the orthogonal projection operator on $\Sc$. Let $A \in \cplx^{n_{\Psi} \times n_{\Psi}}$ be a matrix such that $\Pc_{\Sc} \Kaug \Psi = A \Psi$ (this matrix exists since $\Sc$ is always invariant under $\Pc_{\Sc} \Kaug$). Decompose $A$ in blocks as follows
\begin{align}\label{eq:A-Kaug}
A = \begin{bmatrix}
A_{11} & A_{12}
\\
A_{21} & A_{22} 
\end{bmatrix},
\end{align}
where $A_{11} \in \cplx^{n_{\Hc} \times n_{\Hc}}$. Then, 
\begin{enumerate}
\item for all $h = v_{h}^* H \in \Hc$, with $v_h \in \cplx^{n_{\Hc}}$, define
\begin{align}\label{eq:best-KCF-predictor}
\Pf_{h \circ \Tc}:=  v_{h}^* (A_{11} I_{\Uc} + A_{12} G) H \in \Sc. 
\end{align}
Then, $\Pf_{h \circ \Tc}$ is the best approximation (orthogonal projection) of $h \circ \Tc$ on $\Sc$, i.e., $h \circ \Tc \bestapprox{\Sc} \Pf_{h \circ \Tc}$.
\item $H \circ \Tc \bestapprox{\Sc} (A_{11} I_{\Uc} + A_{12} G) \; H$, where $\bestapprox{\Sc}$ is defined element-wise. \oprocend
\end{enumerate}
\end{theorem}

Theorem~\ref{t:optimal-approx-Kaug} provides a practical way to find best approximations of input-state separable models for KCF via approximations of the action of a single linear operator.

The predictors $\Pf_{h \circ \Tc}$ in Theorem~\ref{t:optimal-approx-Kaug} only depend on the choice of normal space $\Sc$ and not on the choice of normal basis $\Psi$, as detailed below. 

\begin{lemma}\longthmtitle{Invariance of Function Predictors in Theorem~\ref{t:optimal-approx-Kaug} under choice of Normal Basis}\label{l:predictor-does-not-depend-on-basis}
Let $\Pf_{h \circ \Tc}$ and $\bar{\Pf}_{h \circ \Tc}$ be predictors of $h \circ \Tc$ with respect to normal bases $\Psi$ and $\bar{\Psi}$ with $\Sc = \Span(\Psi) = \Span(\bar{\Psi})$. Then, $\Pf_{h \circ \Tc} = \bar{\Pf}_{h \circ \Tc}$. \oprocend
\end{lemma}
\smallskip
The proof follows from the fact that $\Pf_{h \circ \Tc}$ and $\bar{\Pf}_{h \circ \Tc}$ are orthogonal projections of $h \circ \Tc$ on $\Sc$; hence, the result trivially follows from the uniqueness of orthogonal projections\footnote{Note that even if the inner-product space is not complete, the uniqueness here still holds given that all elements of the problem can be contained in a finite-dimensional subspace isomorphic to $\cplx^d$ for some $d \in \naturals$.}.

\section{Data-Driven Setup and Problem Statement}\label{sec:problem-setup}
Koopman operator-based methods are widely used in data-driven settings since they utilize the computationally and theoretically attractive regular structure of linear maps acting on vector spaces. 
With this motivation, we study the accuracy of optimal approximate models in Theorem~\ref{t:optimal-approx-Kaug} in a data-driven setting. To achieve this goal, we first provide a proper data-driven setup and then introduce our main problem.

\subsection{Data-Driven Setup}\label{sec:data}
To apply Theorem~\ref{t:optimal-approx-Kaug}, we need to provide a notion of data from system~\eqref{eq:control-system}, explain how one can evaluate the normal basis~\eqref{eq:normal-basis} on the data, and provide a useful notion of inner-product on $\Faug$ that is amenable to the data-driven case.

\textbf{Data:} let matrices $X^+,X \in \real^{n \times N}$, and $U \in \real^{m \times N}$ be such that $x_i^+ = \Tc(x_i, u_i)$ for $i \in \until{N}$, where $x_i^+$, $x_i$ and $u_i$ are the $i$th columns of $X^+$, $X$, and $U$ respectively.

\textbf{Application of Bases on Data:} given that the input and state spaces of~\eqref{eq:control-system} are subsets of $\real^m$ and $\real^n$, without loss of generality, throughout the paper we use real-valued basis functions to simplify the computations\footnote{$\Span(\Psi)$ still contains complex functions since it is defined on field $\cplx$.}. We use the following notation for application of normal basis $\Psi$ and its component $H$ in~\eqref{eq:normal-basis} on data matrices
\begin{align*}
&\Psi(X,U) := [\Psi(x_1, u_1), \ldots, \Psi(x_N,u_N)] \in \real^{n_{\Psi} \times N},
\\
&H(X) := [H(x_1),  \ldots, H(x_N)] \in \real^{n_{\Hc} \times N}.
\end{align*}
We use the same convention for $\Psi(X^+,U)$ and $H(X^+)$.

\textbf{Choice of Inner-Product Space:} consider the data matrix $Z = [X^T, U^T]^T$ and define the empirical measure $\mu_{Z} = \frac{1}{N} \sum_{i=1}^N \delta_{z_i}$, where $z_i$ is the $i$th column of $Z$ and $\delta_{z_i}$ is the Dirac measure at point $z_i$. We set $\Faug$ to be space $L_2$ under the empirical measure $\mu_{Z}$. Under this choice, the computation of matrix $A$ in~\eqref{eq:A-Kaug} boils down to the well-known Extended Dynamic Mode Decomposition (EDMD) method~\cite{MOW-IGK-CWR:15} applied to system~\eqref{eq:Taug} and basis~\eqref{eq:normal-basis} as
\begin{align}\label{eq:A-edmd}
A = \arg\min_{K} \| \Psi(X^+, U) - K  \Psi(X, U)\|_F.
\end{align}

Throughout the paper, we adopt the following assumption.
\begin{assumption}\longthmtitle{Full Row Rank Dictionary Matrices}\label{a:rank}
The matrices $\Psi(X,U)$ and $H(X^+)$ have full row rank. \oprocend
\end{assumption}

Assumption~\ref{a:rank}  implies that the elements of $\Psi$ must be linearly independent (i.e., they form a basis for $\Span (\Psi)$) and the data needs to be diverse enough to distinguish between the elements of $\Psi$ and $H$. Notably, under Assumption~\ref{a:rank}, the solution of EDMD optimization~\eqref{eq:A-edmd} is unique and can be computed in closed-form by
\begin{align}\label{eq:A-edmd-closed-form}
A = \Psi(X^+,U) \, \Psi(X,U)^\dagger.
\end{align}

\subsection{Problem Statement}
An important aspect of any modeling procedure is the ability to assess the accuracy of the resulting model for purposes of: (i) verification and (ii) the possibility of using the accuracy measure as a loss function in data-driven model learning.  It is therefore imperative to develop data-driven error bounds for the optimal function approximations in Theorem~\ref{t:optimal-approx-Kaug}, with the properties detailed below\footnote{The work~\cite{MH-JC:23-csl} considers a related problem for systems without inputs, where the best data-driven approximations are the EDMD predictors. Our aim here is to extend the notion introduced in~\cite{MH-JC:23-csl} to control systems.}:

\begin{problem}\longthmtitle{Characterizing the Accuracy of Input-State Separable Predictors}\label{prob:accuracy-measure}
Given the normal space $\Sc$ with a normal basis $\Psi$ (cf.~\eqref{eq:normal-basis}) and data matrices $X^+$, $X$, and $U$, we aim to provide an accuracy measure that
\begin{enumerate}
\item bounds the distance between $h \circ \Tc$ and its optimal predictor $\Pf_{h \circ \Tc}$ (cf.~Theorem~\ref{t:optimal-approx-Kaug}) for all $h \in \Hc$;
\item only depends on $\Sc$ and the data, and is invariant under the choice of basis (to be compatible with Lemma~\ref{l:predictor-does-not-depend-on-basis});
\item can be computed in closed-form to be used as loss function in optimization solvers. \oprocend
\end{enumerate}

\end{problem}

\section{Control Forward-Backward Consistency}
Before addressing Problem~\ref{prob:accuracy-measure}, we point out  the fact that only the top part of matrix $A$ ($A_{11}$ and $A_{12}$) in~\eqref{eq:A-Kaug} is used to construct the predictors in Theorem~\ref{t:optimal-approx-Kaug}. Therefore, if the goal is to only construct an input-state separable model, one does not need to compute the entire matrix $A$.
\begin{lemma}\longthmtitle{Computing the Top Block of~\eqref{eq:A-Kaug} in Data-Driven Setting}\label{l:A-top}
Let $A_{11}$ and $A_{12}$ be the top blocks of matrix $A$ in~\eqref{eq:A-edmd-closed-form} according to the partitioning in~\eqref{eq:A-Kaug}. Then, 
\begin{align*}
[A_{11}, A_{12}]
& = \arg\min_{P} \left \| H(X^+) - P\,\Psi(X,U)\right\|_{F} 
\nonumber \\
&= H(X^+) \Psi(X,U)^\dagger. \eqoprocend
\end{align*}
\end{lemma}
\smallskip
The proof follows trivially from directly comparing the closed-form solutions of the optimization problems in Lemma~\ref{l:A-top} and \eqref{eq:A-edmd-closed-form} and is omitted for space reasons. Lemma~\ref{l:A-top} provides a much smaller regression problem to solve, in order to build optimal input-state separable models and function predictors in Theorem~\ref{t:optimal-approx-Kaug}. 

Next, we take our first step towards addressing Problem~\ref{prob:accuracy-measure}. Note that the source of error in the approximate input-state separable models and function predictors come from the approximation error in Lemma~\ref{l:A-top}: if  
\begin{align*}
\min_{P} \left \| H(X^+) - P\,\Psi(X,U)\right\|_{F}=0,
\end{align*}
or equivalently $\row(H(X^+))\subseteq \row(\Psi(X,U))$,
the function predictors and models are exact. Otherwise, there will be an error due to this subspace mismatch. Therefore, inspired by the work~\cite{MH-JC:23-csl}, we measure this deviation using a forward-backward regression notion\footnote{
We point out a major difference in terminology with~\cite{MH-JC:23-csl}: the notion of forward-backward regressions in~\cite{MH-JC:23-csl} coincides with running the system forward and backward in time.  However, in control systems the backward regression is not connected to running the system backward in time.
}.

\begin{definition}\longthmtitle{Control Forward-Backward Consistency Matrix}\label{def:Mcc}
Given the elements of Problem~\ref{prob:accuracy-measure},
let $A_f \in \real^{n_{\Hc} \times n_{\Psi}}$ and $A_b  \in \real^{n_{\Psi} \times n_{\Hc}}$ be the solutions of the following forward and backward least-squares problems
\begin{align*}
A_f = \arg\min_{P} \| H(X^+) - P \Psi(X,U) \|_F = H(X^+) \Psi(X,U)^\dagger,
\nonumber \\
A_b = \arg\min_{Q} \| \Psi(X,U) - Q H(X^+) \|_F = \Psi(X,U) H(X^+)^\dagger.
\end{align*}
We define the control forward-backward consistency matrix (control consistency matrix for short) as
\begin{align}\label{eq:control-consistency}
M_{CC} = I - A_f A_b. \eqoprocend
\end{align}
\end{definition}
\smallskip

Note that if $\row(H(X^+)) \subseteq \row(\Psi(X,U))$, we have $M_{CC} = 0$, otherwise, it would be non-zero. Next, we state a few basic properties of the control consistency matrix.

\begin{proposition}\longthmtitle{Properties of Control Consistency Matrix}\label{p:Mcc-properties}
Given Assumption~\ref{a:rank}, the control consistency matrix has the following properties:
\begin{enumerate}
\item it is similar to a symmetric matrix;
\item its eigenvalues belong to the interval $[0,1]$.
\end{enumerate}
\end{proposition}
\begin{proof}
For brevity, we use the notation $J = H(X^+)$ and $L=\Psi(X,U)$ throughout this proof.

(a) Using the closed-form solutions for $A_f$ and $A_b$ in Definition~\ref{def:Mcc} and given Assumption~\ref{a:rank}, one can write
\begin{align}\label{eq:consistency-expanded-1}
M_{CC} = I - J L^\dagger L J^\dagger = I - J L^\dagger L J^T (J J^T)^{-1},
\end{align}
where we have used $J^\dagger = J^T (J J^T)^{-1}$. Let $R = (J J^T)^{\frac{1}{2}}$ and note that $R$ exists since $JJ^T$ is a symmetric positive definite matrix. Then, we can define
\begin{align}\label{eq:similarity-transform}
\tilde{M}_{CC} &= R^{-1} M_{CC} R = R^{-1} R - R^{-1} J L^\dagger L J^T (J J^T)^{-1} R 
\nonumber \\
&=  I - (J J^T)^{-\frac{1}{2}} J L^\dagger L J^T (J J^T)^{-\frac{1}{2}}.
\end{align}
Let $S = J^T (J J^T)^{-\frac{1}{2}}$. Since $(J J^T)^{-\frac{1}{2}}$ is symmetric
\begin{align}\label{eq:similarity-transform-2}
\tilde{M}_{CC} = R^{-1} M_{CC} R = I - S^T L^\dagger  L S.
\end{align}  
The proof follows from the fact that $L^\dagger  L $ is symmetric.

(b) Since similarity transformations preserve the eigenvalues, based on~\eqref{eq:similarity-transform-2}, we have $\spec(\tilde{M}_{CC}) = \spec(M_{CC})$. Moreover, since $\tilde{M}_{CC}$ is symmetric, its eigenvalues are real and its eigenvectors are orthogonal. Now, let $v \in \real^{n_{\Hc}}\setminus \{0\}$ be an eigenvector of $\tilde{M}_{CC}$ with eigenvalue $\lambda \in \real\setminus \{0\}$, i.e., $\tilde{M}_{CC} v = \lambda v$. Multiplying both sides from the left by $v ^T$ and using~\eqref{eq:similarity-transform-2}, gives
\begin{align*}
\lambda v^T v = v^T \tilde{M}_{CC} v = v^T v - v^T S^T L^\dagger L S v,
\end{align*}
where $S = J^T (J J^T)^{-\frac{1}{2}}$ and we have $S^T S = I$; hence, by defining $w = Sv$, we can write the previous equation as
\begin{align}\label{eq:similarity-transform-3}
\lambda w^T w = w^T w - w L^\dagger L w = w^T (I-L^\dagger L) w.
\end{align}
Moreover, note that matrix $I-L^\dagger L$ projects orthogonally onto the null space of $L$ and is symmetric. Using the basic property of pseudoinverse (that $L L^\dagger L = L$) we find that $(I-L^\dagger L)^T (I-L^\dagger L) = (I-L^\dagger L)^2 = I-L^\dagger L$.  Using this property in conjunction with~\eqref{eq:similarity-transform-3}, one can write $\lambda \| w \|^2 =  w^T (I-L^\dagger L)^T (I-L^\dagger L) w = \| (I-L^\dagger L) w \|^2$, leading to
\begin{align}\label{eq:eigenvalue-ratio}
\lambda  =  \frac{\| (I-L^\dagger L) w \|^2}{\| w\|^2}.
\end{align}
From~\eqref{eq:eigenvalue-ratio}, clearly $\lambda \geq 0$. Moreover, since $(I-L^\dagger L)$ is an orthogonal projection, we have $\| (I-L^\dagger L) w \| \leq \| w\|$. By~\eqref{eq:eigenvalue-ratio}, this implies $\lambda \leq 1$, concluding the proof.
\end{proof}

It is worth noting that Proposition~\ref{p:Mcc-properties} implies that $M_{CC}$ is similar to a positive semi-definite matrix.
Next, we show that the spectrum of $M_{CC}$ in invariant under the choice of normal basis, a property  we sought in Problem~\ref{prob:accuracy-measure}.

\begin{proposition}\longthmtitle{$\spec(M_{CC})$ is Invariant under the Choice of Normal Basis}\label{p:consistency-spectrum-invariance}
Let $\Psi$ and $\bar{\Psi}$ be two normal bases for the normal space $\Sc$, written as
\begin{align*}
\Psi = \begin{bmatrix} I_{\Uc} \\ G \end{bmatrix} H , \quad \bar{\Psi} = \begin{bmatrix} I_{\Uc}  \\ \bar{G} \end{bmatrix} \bar{H}.
\end{align*}
Given data matrices $X,X^+, U$ and Assumption~\ref{a:rank}, let $M_{CC}$ and $\bar{M}_{CC}$ be control consistency matrices given bases $\Psi$ and $\bar{\Psi}$ respectively. Then, $\spec(M_{CC}) = \spec(\bar{M}_{CC})$.
\end{proposition}
\begin{proof}
For brevity, we use the notation $J = H(X^+)$, $\bar{J} = \bar{H}(X^+)$, $L=\Psi(X,U)$, and $\bar{L}= \bar{\Psi}(X,U)$.

Since both $\Psi$ and $\bar{\Psi}$ are bases for the same vector space, there exists a nonsingular matrix $R$ such that $\bar{\Psi} = R \Psi$.  Decompose $R$ in blocks with appropriate sizes corresponding to the decomposition of $\Psi$ as
\begin{align}\label{eq:R-decompose2}
R = \begin{bmatrix} R_{11} & R_{12}
\\
R_{21} & R_{22}
\end{bmatrix}.
\end{align}
Based on Lemma~\ref{l:change-normal-bases} in the appendix, we have $R_{12} = 0$ and $R_{11}$ is invertible. Hence, $\bar{J} = R_{11} J$, and one can write
\begin{align}\label{eq:Mcc-change-of-basis-1}
\bar{M}_{CC} = I - \bar{J} \bar{L}^\dagger \bar{L} \bar{J}^\dagger = R_{11} R_{11}^{-1} - R_{11} J (L^\dagger L) (R_{11} J)^\dagger ,
\end{align}
where in the last equality we have used $\bar{L}^\dagger \bar{L} = L^\dagger L$ since these matrices are orthogonal projections on the same subspace and therefore equal as a result of the uniqueness of orthogonal projections. 
Next, noting that $\bar{J}$ has full row rank and $R_{11}$ is nonsingular (cf. Lemma~\ref{l:change-normal-bases}), one can write
\begin{align*}
&(R_{11} J)^\dagger = (R_{11} J)^T \big( (R_{11} J) (R_{11} J)^T \big)^{-1}
\\
&= J^T R_{11}^T (R_{11} J J^T R_{11}^T)^{-1} = J^T (J J^T )^{-1} R_{11}^{-1} = J^\dagger R_{11}^{-1}.
\end{align*}
Using the previous equality in conjunction with~\eqref{eq:Mcc-change-of-basis-1}, gives
\begin{align*}
\bar{M}_{CC} = R_{11} (I -  J (L L^\dagger) J^\dagger ) R_{11}^{-1} = R_{11} M_{CC} R_{11}^{-1}.
\end{align*}
The proof concludes by noting that similarity transformations do not change the eigenvalues.
\end{proof}

Proposition~\ref{p:consistency-spectrum-invariance} shows that the eigenvalues of $M_{CC}$ are an intrinsic property of the underlying subspace and the problem data, and do not depend on
a choice of basis. The spectral radius of $M_{CC}$ which coincides with its largest eigenvalue following Proposition~\ref{p:Mcc-properties}, measures how close $M_{CC}$ is to the zero matrix and plays a central role in the rest of the paper. Hence, we provide the following definition.

\begin{definition}\longthmtitle{Control Consistency Index (CCI)}
We define the \emph{control consistency index} denoted by $\Ic_{CC}$ as the largest eigenvalue of $M_{CC}$, i.e., $\Ic_{CC}= \lambda_{\max}(M_{CC})$. \oprocend
\end{definition}

Next, we study how the control consistency index can provide an accuracy measure for the quality of input-state separable models and the corresponding function predictors.

\section{The Control Consistency Index Sharply Bounds the Predictors' Accuracy}
Our main result shows that the {\em control consistency index} (CCI) tightly bounds the worst-case relative root mean square error (RRMSE) of function approximations.

\begin{theorem}\longthmtitle{RRMSE of Prediction is Sharply Bounded by the CCI}\label{t:rrmse-bound}
Consider the data-driven setting in Section~\ref{sec:data} and Assumption~\ref{a:rank}. Given $h \in \Hc$, let $h \circ \Tc$ be its evolution under~\eqref{eq:control-system}, and let $\Pf_{h \circ \Tc}$ be its KCF predictor (cf.~Theorem~\ref{t:optimal-approx-Kaug}). Define the relative root mean square error for the prediction of $h \circ \Tc$ on the data as
\begin{align}\label{eq:rrmse-def}
\RRMSE_{h \circ \Tc} &= \sqrt{ \frac{\frac{1}{N} \sum_{i=1}^N |h\circ \Tc(x_i,u_i) - \Pf_{h \circ \Tc}(x_i,u_i)|^2}{\frac{1}{N} \sum_{i=1}^N |h \circ \Tc (x_i,u_i)|^2} }
\nonumber \\
&=\sqrt{ \frac{\sum_{i=1}^N |h(x_i^+) - \Pf_{h \circ \Tc}(x_i,u_i)|^2}{ \sum_{i=1}^N |h(x_i^+)|^2}}.
\end{align}
Then, $\RRMSE_{\max} := \max_{\substack{h \in \Hc \\ h \neq 0}} \RRMSE_{h \circ \Tc} = \sqrt{\Ic_{CC}}$.
\end{theorem}
\begin{proof}
For brevity, we use the notation $J = H(X^+)$ and $L = \Psi(X,U)$ throughout the proof.
Given $h \in \Hc$ with representation $h = v_h^* H$, the identity $h \circ \Tc(x_i,u_i) = h(x_i^+)$, the definition of $\Pf_{h \circ \Tc}$ in Theorem~\ref{t:optimal-approx-Kaug}, and closed-form solution in Lemma~\ref{l:A-top}, one can use~\eqref{eq:rrmse-def} to write
\begin{align}\label{eq:rrmse-squared}
\RRMSE_{h \circ \Tc}^2 = \frac{ \| v_h^* J - v_h^* J L^\dagger L \|^2 }{ \|v_h^* J \|^2 }.
\end{align}
Noting that the map $v_h \in \cplx^{\dim(\Hc)} \leftrightarrow h \in \Hc$ is bijective, one can use~\eqref{eq:rrmse-squared} and the definition of $\RRMSE_{\max}$, to write
\begin{align}\label{eq:rrmse-max-v}
\RRMSE_{\max}^2 = \max_{\substack{v \in \cplx^{n_{\Hc}} \\ v \neq 0}} \frac{ \| v^* J - v^* J L^\dagger L \|^2 }{ \|v^* J \|^2 }.
\end{align}
Consider the matrix $\bar{J} \in \real^{n_{\Hc} \times N}$ whose rows form an orthonormal basis for $\row(J)$. Assumption~\ref{a:rank} implies that $\bar{J}$ always exists.
Since for all $v \in \cplx^{n_{\Hc}}$ there exists a unique $w \in \cplx^{n_{\Hc}}$ such that $v^* J = w^* \bar{J}$, we can
rewrite~\eqref{eq:rrmse-max-v} as
\begin{align}\label{eq:rrmse-max-w}
\RRMSE_{\max}^2 &= \max_{\substack{w \in \cplx^{n_{\Hc}} \\ w \neq 0}} \frac{ \| w^* \bar{J} - w^* \bar{J} L^\dagger L \|^2 }{ \| w^* \bar{J} \|^2 }
\nonumber \\
&= \max_{\substack{w \in \cplx^{n_{\Hc}} \\ w \neq 0}} \frac{ \| w^* \bar{J} - w^* \bar{J} L^\dagger L \|^2 }{ \| w^* \|^2 },
\end{align}
where in the second equality we have used the orthonormality of the rows of $\bar{J}$, i.e., $\bar{J} \bar{J}^T = I$.
Moreover, using the same identity, the numerator of~\eqref{eq:rrmse-max-w} can be expanded as
\begin{align}\label{eq:numerator-expansion}
\| w^* \bar{J} - w^* \bar{J} L^\dagger L \|^2 = w^* w - 2 w^* \bar{J} L^\dagger L \bar{J}^T w
\nonumber \\
 + w^* \bar{J} L^\dagger L (L^\dagger L)^T \bar{J}^T w
= w^* w - w^* \bar{J} L^\dagger L \bar{J}^T w
\nonumber \\
= w^* ( I - \bar{J} L^\dagger L \bar{J}^T ) w,
\end{align}
where in the second equality we noted that $L^\dagger L$ is symmetric, and therefore $L^\dagger L (L^\dagger L)^T = L^\dagger L L^\dagger L = L^\dagger L$ (we have used the identity $LL^\dagger L = L$).
Combining~\eqref{eq:rrmse-max-w} and~\eqref{eq:numerator-expansion} yields
\begin{align}\label{eq:rrmse-lambda-max}
\RRMSE_{\max}^2 &= \max_{\substack{w \in \cplx^{n_{\Hc}} \\ w \neq 0}} \frac{ w^* ( I - \bar{J} L^\dagger L \bar{J}^T ) w }{ \|w^*\|^2 }
\nonumber \\
&= \lambda_{\max}(I - \bar{J} L^\dagger L \bar{J}^T).
\end{align}
To complete the proof, we just need to show $\lambda_{\max}(M_{CC}) = \lambda_{\max}(I - \bar{J} L^\dagger L \bar{J}^T)$.
Let $R$ be such that $J = R \bar{J}$. Then,
\begin{align}\label{eq:mcc-similarity}
M_{CC} &= I - J L^\dagger L J^\dagger = I - R \bar{J} L^\dagger L (R \bar{J})^\dagger
\nonumber \\
&= I - R \bar{J} L^\dagger L (R \bar{J})^T [ (R \bar{J}) (R \bar{J})^T ]^{-1}
\nonumber \\
&= R ( I - \bar{J} L^\dagger L \bar{J}^T ) R^{-1},
\end{align}
where in the second equality we used $A^\dagger = A^T (A A^T)^{-1}$ for full row rank matrices, and in the last equality we used the fact that $\bar{J} \bar{J}^T = I$.
Equality~\eqref{eq:mcc-similarity} shows that $M_{CC}$ and $I - \bar{J} L^\dagger L \bar{J}^T$ are similar; therefore, their eigenvalues are equal; hence, $\lambda_{\max}(M_{CC}) = \lambda_{\max}(I - \bar{J} L^\dagger L \bar{J}^T)$.
This, together with~\eqref{eq:rrmse-lambda-max}, concludes the proof.
\end{proof}

\begin{remark}\longthmtitle{The CCI Bounds the Relative $L_2$ Error}
Theorem~\ref{t:rrmse-bound} can be interpreted as a sharp bound on the worst-case relative $L_2$-norm prediction. In fact, given the empirical measure $\mu_{Z}$ defined in Section~\ref{sec:data}, we have
\begin{align*}
\sqrt{\Ic_{CC}} = \max_{\substack{h \in \Hc \\ h \neq 0}} \frac{\| h \circ \Tc - \Pf_{h \circ \Tc} \|_{L_2 (\mu_{Z})}}{\| h \circ \Tc \|_{L_2 (\mu_{Z})}}.
\eqoprocend
\end{align*}
\end{remark}
\smallskip

\section{Implications for Linear and Bilinear Forms}
Lifted linear and bilinear models are subclasses of input-state separable forms~\cite[Lemma~4.6]{MH-JC:26-auto}. This leads to the following natural question:
\begin{quote}
  {\em How can one apply the results presented here 
  to the case of lifted linear and bilinear models?}   
\end{quote}
To answer this question, we next provide forms of the normal basis~\eqref{eq:normal-basis} such that the input-state separable models in Theorem~\ref{t:optimal-approx-Kaug} become lifted linear or bilinear models.

\begin{lemma}\longthmtitle{Choice of Normal Basis for Lifted Linear Models}\label{l:lifted-linear-basis}
In the non-degenerate normal basis~\eqref{eq:normal-basis} set
\begin{align}\label{eq:lifted-linear-defs}
H(x) = \begin{bmatrix} \bar{H}(x) \\ 1_{\Xc}(x) \end{bmatrix}, \quad
G(u) =  [0_{m \times \dim(\bar{H})},  u],
\end{align}
where $1_{\Xc}: \Xc \to \real$ is the constant function equal to one, i.e., $1_{\Xc} (x) =1$ for all $x \in \Xc$.
Then, the input-state separable model in Theorem~\ref{t:optimal-approx-Kaug} turns into
\begin{align*}
  H \circ \Tc \bestapprox{\Sc}  A_{11} I_{\Uc} H + A_{12} (u \, 1_{\Xc}) .  
\end{align*}
 Moreover, given the initial condition $x_0 \in \Xc$, the associated point-wise predictor (cf.~\eqref{eq:input-state-separable-pointwise}) turns into the lifted linear model
\begin{align*}
z^+ = A_{11} z + A_{12} u, \quad \text{with} \quad z_0 = H(x_0). \eqoprocend
\end{align*}
\end{lemma}
\smallskip
The proof directly follows from direct computation and noting that $u \, 1_{\Xc}(x) = u$ for all $x \in \Xc$.

\begin{remark}\longthmtitle{Generality of the Decomposition of $H$ in Lemma~\ref{l:lifted-linear-basis}}
The function $1_{\Xc}$ can always be included in $H$. If $\Span(H)$ contains a constant function, the form in~\eqref{eq:lifted-linear-defs} follows by a change of basis on $H$. If $1_{\Xc} \notin \Span(H)$, one can enlarge $H$ by adding $1_{\Xc}$ to it. This addition never degrades a model's accuracy, and may even improve it, since components in $\Span(1_{\Xc})$ are always predicted exactly: constant functions remain constant along all trajectories.
\oprocend
\end{remark}
\smallskip

\begin{lemma}\longthmtitle{Choice of Normal Basis for Lifted Bilinear Models}\label{l:lifted-bilinear-basis}
In the non-degenerate normal basis~\eqref{eq:normal-basis}, set
\begin{align}\label{eq:bilinear-G-def}
G(u) = [u_1 I, u_2 I,  \ldots,  u_m I ]^T,
\end{align}
where $I$ is the identity matrix of size $n_{\Hc}$. Decompose $A_{12} \in \cplx^{n_{\Hc} \times m\, n_{\Hc}}$ in~\eqref{eq:A-Kaug} as $A_{12} = [B_1, B_2, \ldots, B_m]$ with $B_i \in \cplx^{n_{\Hc} \times n_{\Hc}}$.
Then, the input-state separable model in Theorem~\ref{t:optimal-approx-Kaug} turns into
\begin{align*}
  H \circ \Tc \bestapprox{\Sc}  (A_{11} I_{\Uc} + \sum_{i=1}^m u_i B_i) H .  
\end{align*}
Also, given the initial condition $x_0 \in \Xc$, the associated point-wise predictor (cf.~\eqref{eq:input-state-separable-pointwise}) turns into the lifted bilinear model
\begin{align*}
	z^+ = A_{11} z + \sum_{i=1}^m u_i B_i z, \quad \text{with} \quad z_0 = H(x_0). \eqoprocend
\end{align*}
\end{lemma}
\smallskip
The proof follows by direct computation. Next, we explain how Theorem~\ref{t:rrmse-bound} applies to lifted linear and bilinear forms.

\begin{corollary}\longthmtitle{Theorem~\ref{t:rrmse-bound} Applies to Lifted Linear and Bilinear Forms}
Compute $M_{CC}$ in Definition~\ref{def:Mcc} via the additional structure imposed in Lemma~\ref{l:lifted-linear-basis} or Lemma~\ref{l:lifted-bilinear-basis} on basis $\Psi$. Then, Theorem~\ref{t:rrmse-bound}  applies to $\Pf_{h \circ \Tc}$ (which in this case has an added linear or bilinear structure). \oprocend
\end{corollary}

\section{Optimization-Based Model Learning}\label{sec:learning}
Given a fixed subspace $\Span(\Psi)$ and data matrices $X, X^+, U$, Theorem~\ref{t:optimal-approx-Kaug} provides the optimal KCF model and predictors, while Theorem~\ref{t:rrmse-bound} establishes a tight worst-case bound for these predictors. To find a better model given a fixed dataset, one must change $\Psi$, which leads to the following question: how should we choose the normal basis $\Psi$ to obtain a model with low RRMSE error?

The control consistency index is a natural cost function for optimization-based schemes. Let $\Psi^{\phi, \theta}$ represent a parametric family (e.g., neural networks, etc) of normal bases in the form~\eqref{eq:normal-basis} with components $H^\phi: \Xc \to \real^{n_\Hc}$ and $G^{\theta}: \Uc \to \real^{(n_\Psi - n_\Hc) \times n_\Hc}$, for parameters $\phi$ and $\theta$. The corresponding optimization is equivalent to a minimax problem
\begin{align}\label{eq:robust-optimization}
\min_{\substack{\phi, \theta}} \sqrt{\Ic_{CC}^{{\phi, \theta}}}  \; \Leftrightarrow \; \min_{\substack{\phi, \theta}} \max_{\substack{h \in \Hc^\phi \\ h \neq 0}}  \frac{\| h \circ \Tc - \Pf_{h \circ \Tc}^{\phi,\theta} \|_{L_2 (\mu_{Z})}}{\| h \circ \Tc \|_{L_2 (\mu_{Z})}},
\end{align}
where $\Ic_{CC}^{{\phi, \theta}}$ and $\Pf_{h \circ \Tc}^{\phi,\theta}$ are the control consistency index and KCF predictor for basis $\Psi^{\phi, \theta}$. Optimization~\eqref{eq:robust-optimization} considers all functions $h \in \Hc^\phi$ and is insensitive to the choice of basis. 

\textit{Numerical Sensitivity and Computational Efficiency:}
Computing the maximum eigenvalue of a matrix is computationally intensive 
and sensitive to numerical errors when the gap between eigenvalues is small.
To address this, we use $\tr(M_{CC})$ as a proxy. From Proposition~\ref{p:Mcc-properties}(b), we have
\begin{align*}
0 \leq \frac{1}{n_\Hc} \tr(M_{CC}) \leq \Ic_{CC} \leq \tr(M_{CC}).
\end{align*}
Hence, $\tr(M_{CC})$ is equivalent to $\Ic_{CC}$ up to a multiplicative constant and can be used as a cheap surrogate of $\Ic_{CC}$. Moreover, similarly to $\Ic_{CC}$, $\tr(M_{CC})$ is also invariant under the choice of normal basis directly following Proposition~\ref{p:consistency-spectrum-invariance}.

\begin{corollary}\longthmtitle{$\tr(M_{CC})$ is Invariant under the Choice of Normal Basis}
Consider $M_{CC}$ and $\bar{M}_{CC}$ as defined in Proposition~\ref{p:consistency-spectrum-invariance}. Then, $\tr(M_{CC}) = \tr(\bar{M}_{CC})$. \oprocend
\end{corollary}
While both $\Ic_{CC}$ and $\tr(M_{CC})$ are invariant under the choice of basis\footnote{It is tempting to use $\| M_{CC}\|_F$ as a loss function; however, the Frobenius norm is not invariant under the choice of basis.},  computational accuracy (especially in low precision deep learning models) depends on the matrices' condition number. Thus, it is useful to add a small penalty on the condition number of $\Psi^{\phi,\theta}(X,U)$ and $H^\phi(X^+)$.

\textit{Differences with Robust Learning in~\cite{MH-JC:26-auto}:}
The approach in~\cite{MH-JC:26-auto} applies forward-backward consistency for systems without input~\cite{MH-JC:23-csl} to $\Kaug$ (cf.~\eqref{eq:Kaug}), yielding a conservative prediction error bound~\cite[Theorem~8.7(d)]{MH-JC:26-auto}, while the bound in Theorem~\ref{t:rrmse-bound} is tight. Moreover, learning in~\cite{MH-JC:26-auto} relies on the invariance proximity of $\Span(\Psi^{\phi, \theta})$ under $\Kaug$, whereas this work only requires predictor accuracy for $h \in \Hc$. This difference permits optimizing over much smaller subspaces and using smaller parametric families.

\section{Simulations}
Consider the DC motor adopted from~\cite{MH-JC:26-auto}\footnote{A simpler version of this example has been used in~\cite{MK-IM-automatica:18} which was adopted from the experimental study~\cite{SDB-HU:98}.}
\begin{align}\label{eq:DC-motor}
\dot{x}_1 &= -39.3153\,x_1 - 0.805732\,x_2 f(u) + 191.083, \nonumber\\
\dot{x}_2 &= -1.65986\,x_2 + 57.3696\,x_1 f(u) - 333.333,
\end{align}
where $x=[x_1,x_2]^T \in \Xc =[-5,15] \times [-250,125]$ and $u \in \Uc= [-2,2]$. This system is specifically chosen to evaluate how different input nonlinearities impact the accuracy of input-state separable, lifted linear, and bilinear models. To this end, we study two input terms: $f(u) = 2\tanh(u)$ and $f(u) = 2 \tanh(u \cos(u))$. All models are learned for the discretization of~\eqref{eq:DC-motor} with a sampling time of $\Delta t = 5$ms.

\textit{Data:} We sample data (with $\Delta t$) from one $50$s experiment starting at $x = 0$. We apply a piecewise-constant sequence of inputs that are randomly drawn from $\Uc$ and held constant for $0.2$s. This yields $10^4$ snapshots in data matrices $X^+, X, U$, which we divide into $80\%$ training and $20\%$ test sets.

\textit{Parametric Families:} Following Section~\ref{sec:learning}, we set $n_\Hc = 5$ and $n_\Psi = 10$. The function $H^\phi: \Xc \to \real^{n_\Hc}$ is a feedforward neural network with four hidden layers, each containing $64$ neurons with exponential linear unit (ELU) activations. All layers are equipped with layer normalization~\cite{JLB-JRK-GEH:16}. We fix the first two output elements of $H^\phi$ to be the system states. The network $G^\theta: \Uc \to \real^{(n_\Psi - n_\Hc) \times n_\Hc}$ shares the exact same hidden-layer architecture. To learn the lifted linear and bilinear models, we retain $H^\phi$ while fixing the structure of $G$ according to Lemmas~\ref{l:lifted-linear-basis} and~\ref{l:lifted-bilinear-basis}, respectively.

\textit{Loss Function:} Following Section~\ref{sec:learning}, we define our loss function as $\tr(M_{CC}^{\phi, \theta}) + \alpha \; \text{cond}(\Psi^{\phi, \theta}(X,U))$ with $\alpha = 10^{-5}$.

\textit{Training:} We train the $\Psi^{\phi, \theta}$ comprised of $H^\phi$ and $G^\theta$ over batches of size $n_b =100$ via Adam~\cite{DPK-JB:15} with $\beta=(0.9,0.999)$ for 500 epochs. We warm up the learning rate linearly from $10^{-7}$ to $10^{-3}$ during the first 50 epochs, then anneal the rate back to $10^{-7}$ via a cosine schedule thereafter.

\begin{remark}\longthmtitle{Practical Advice}
We offer two implementation guidelines: (i) choose a batch size larger than $n_\Psi$ to ensure the rank condition in Assumption~\ref{a:rank} holds for each batch; (ii) avoid batch normalization--it changes the underlying subspaces per batch, resulting in errors. Layer normalization can be used if needed. \oprocend
\end{remark}

\textit{Results:} Tables~\ref{tab:training-metrics-DC-tanh} and~\ref{tab:training-metrics-DC-tanhcos} show the worst-case errors on the training and test data, computed via Theorem~\ref{t:rrmse-bound}, for the cases $f(u) = 2 \tanh(u)$ and $f(u) = 2 \tanh(u \cos(u))$ in~\eqref{eq:DC-motor}. 
As expected, in both cases, the input-state separable model is significantly more accurate on both the training and test data. 
The bilinear model performs much better for $f(u) = 2 \tanh(u)$ than for the more complex nonlinearity. This disparity arises from the fact that $f(u)=2\tanh(u)$ can be reasonably approximated by a linear map over $\Uc$, rendering the system nearly control-affine, whereas this is not the case for $f(u) = 2 \tanh(u \cos(u))$.

\begin{table}[tb]
	\centering
	\caption{$\RRMSE_{\max}$ for different models given $f(u)= 2 \tanh(u)$ on training and test data gathered from~\eqref{eq:DC-motor}.}
	\label{tab:training-metrics-DC-tanh}
	\begin{tabular}{l
			S[table-format=1.6]
			S[table-format=1.6]}
		\toprule
		\textbf{Model} & {\textbf{$\boldsymbol{\RRMSE_{\max}}$ (train)}} & {\textbf{$\boldsymbol{\RRMSE_{\max}}$ (test)}} \\
		\midrule
		Input-State Separable & 0.002451 & 0.002253 \\
		Lifted Bilinear       & 0.026693 & 0.026992 \\
		Lifted Linear         & 0.078716 & 0.077355 \\
		\bottomrule
	\end{tabular}
\end{table}
\begin{table}[tb]
	\centering
	\caption{$\RRMSE_{\max}$ for different models given $f(u)= 2 \tanh(u \cos(u))$ on training and test data gathered from~\eqref{eq:DC-motor}.}
	\label{tab:training-metrics-DC-tanhcos}
	\begin{tabular}{l
			S[table-format=1.6]
			S[table-format=1.6]}
		\toprule
		\textbf{Model} & {\textbf{$\boldsymbol{\RRMSE_{\max}}$ (train)}} & {\textbf{$\boldsymbol{\RRMSE_{\max}}$ (test)}} \\
		\midrule
		Input-State Separable & 0.004321 & 0.004658 \\
		Lifted Bilinear       & 0.060323 & 0.063704 \\
		Lifted Linear         & 0.081267 & 0.084603 \\
		\bottomrule
	\end{tabular}
    \vskip -0.15 true in
\end{table}

To quantify the point-wise dynamic predictor~\eqref{eq:input-state-separable-pointwise} accuracy and its linear and bilinear special cases, we generate 1000 trajectories of 200 time steps (1s). Initial conditions are sampled from test data, and piecewise constant input values are drawn from $\Uc$. Figure~\ref{fig:DC-motor-long-term-prediction} shows the median and the $25^{th}$--$75^{th}$ percentile range of the relative prediction error for each state variable.  The input-state separable model provides accurate, low-variance predictions for both input nonlinearities:  it successfully captures the true system dynamics. In contrast, while the lifted bilinear model shows reasonable short-horizon accuracy for $f(u)=2\tanh(u)$, the linear model produces inaccurate, high-variance predictions due to missing state-input cross-terms. Furthermore, when $f(u) = 2 \tanh(u \cos(u))$, both lifted linear and bilinear models fail to provide accurate low-variance predictions, implying that strong input nonlinearities require the input-state separable form over its special linear and bilinear cases.

\begin{figure}[tbh]
	\centering 
	{\includegraphics[width=.49\linewidth]{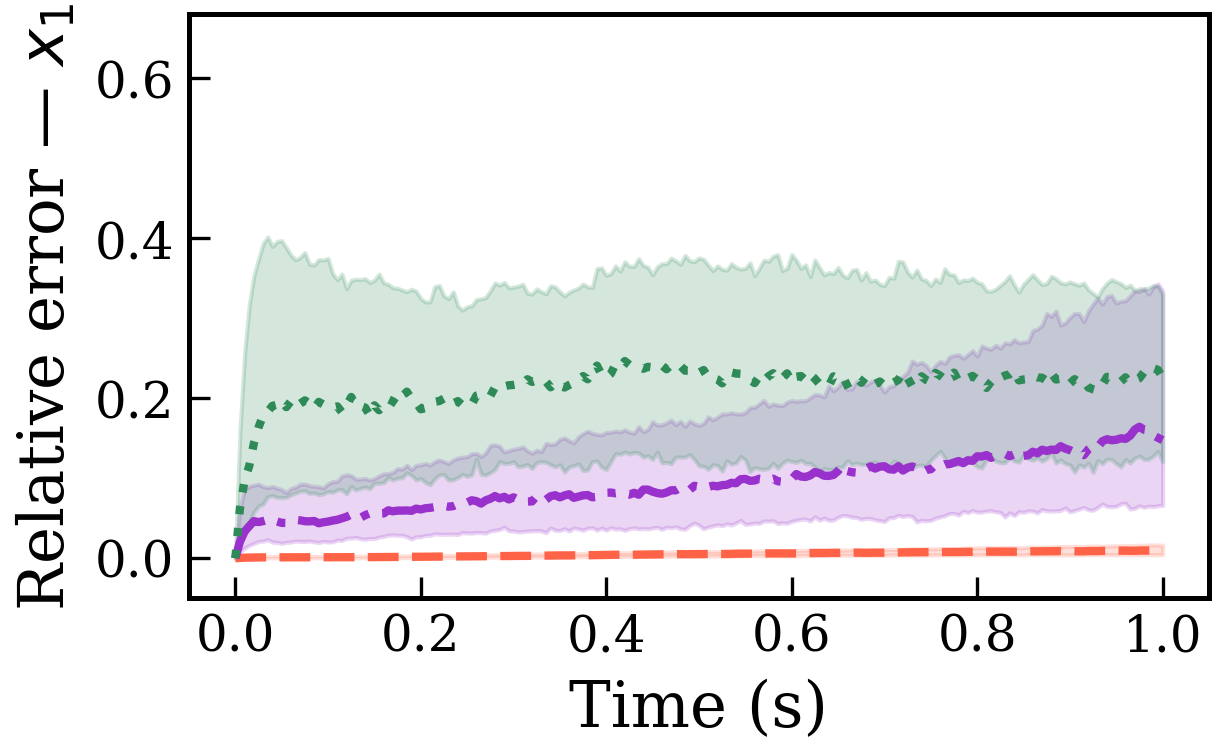}}
	{\includegraphics[width=.49\linewidth]{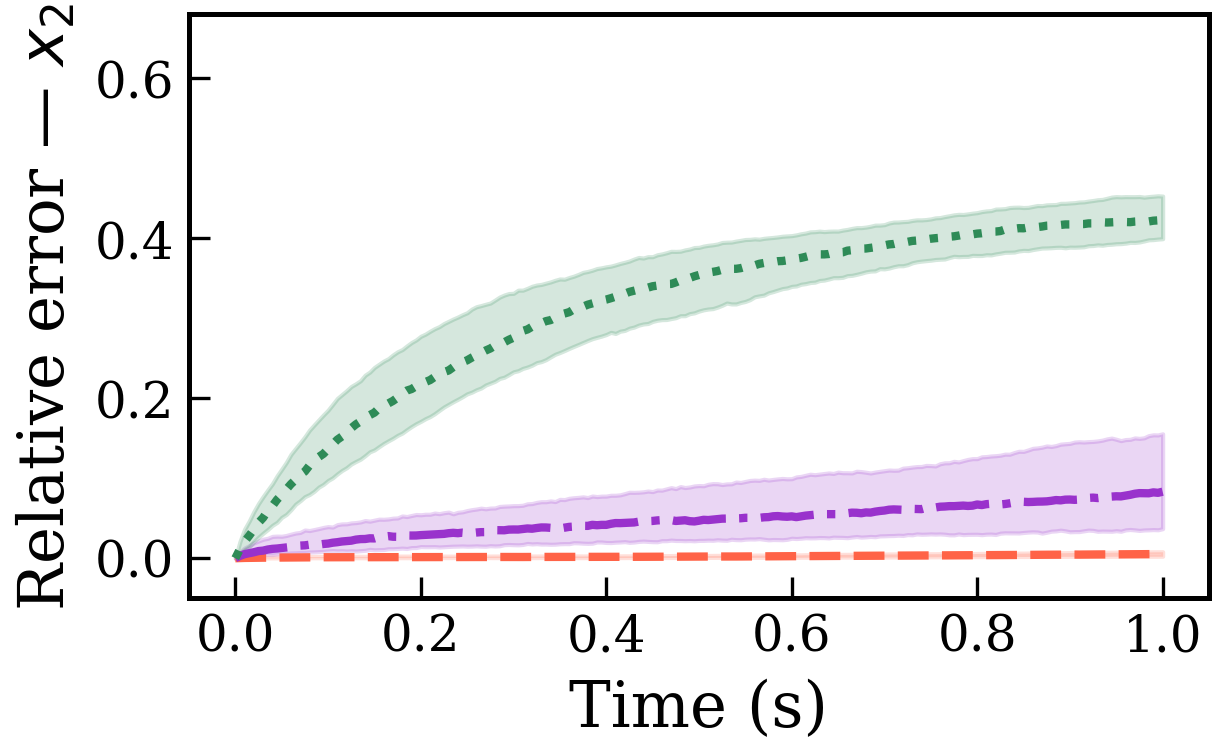}}
	\\
	{\includegraphics[width=.48\linewidth]{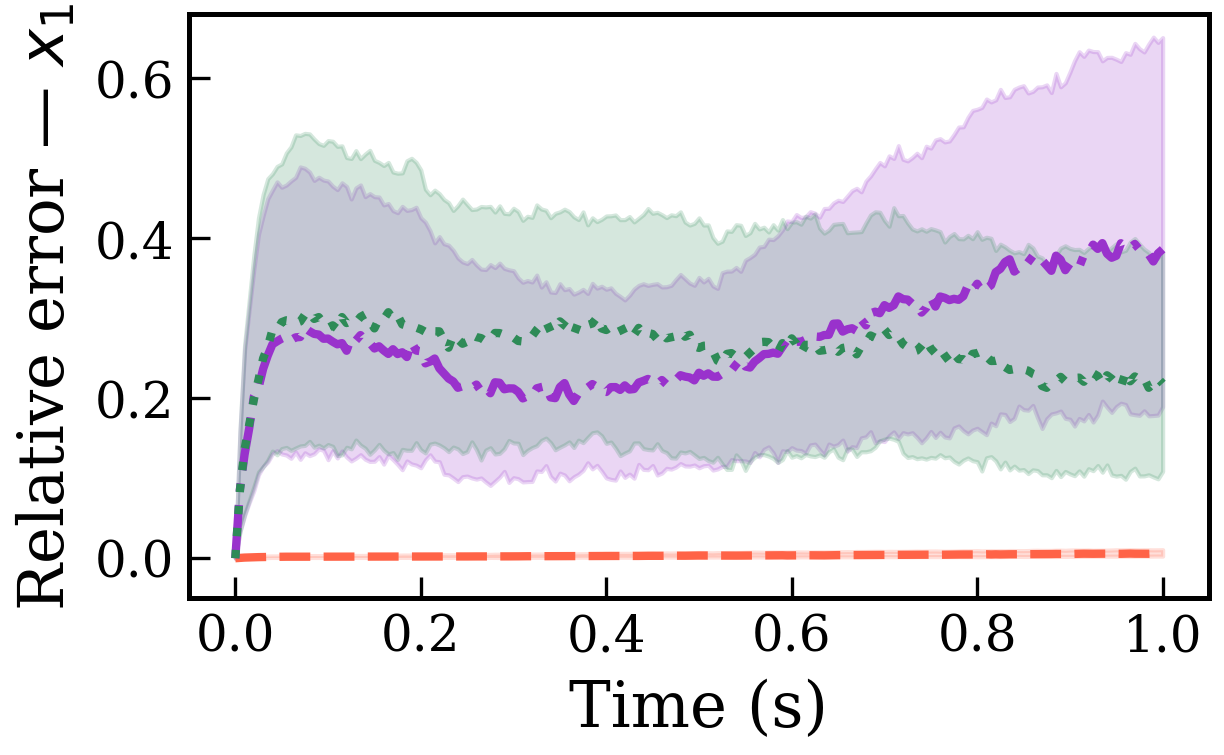}}
	{\includegraphics[width=.49\linewidth]{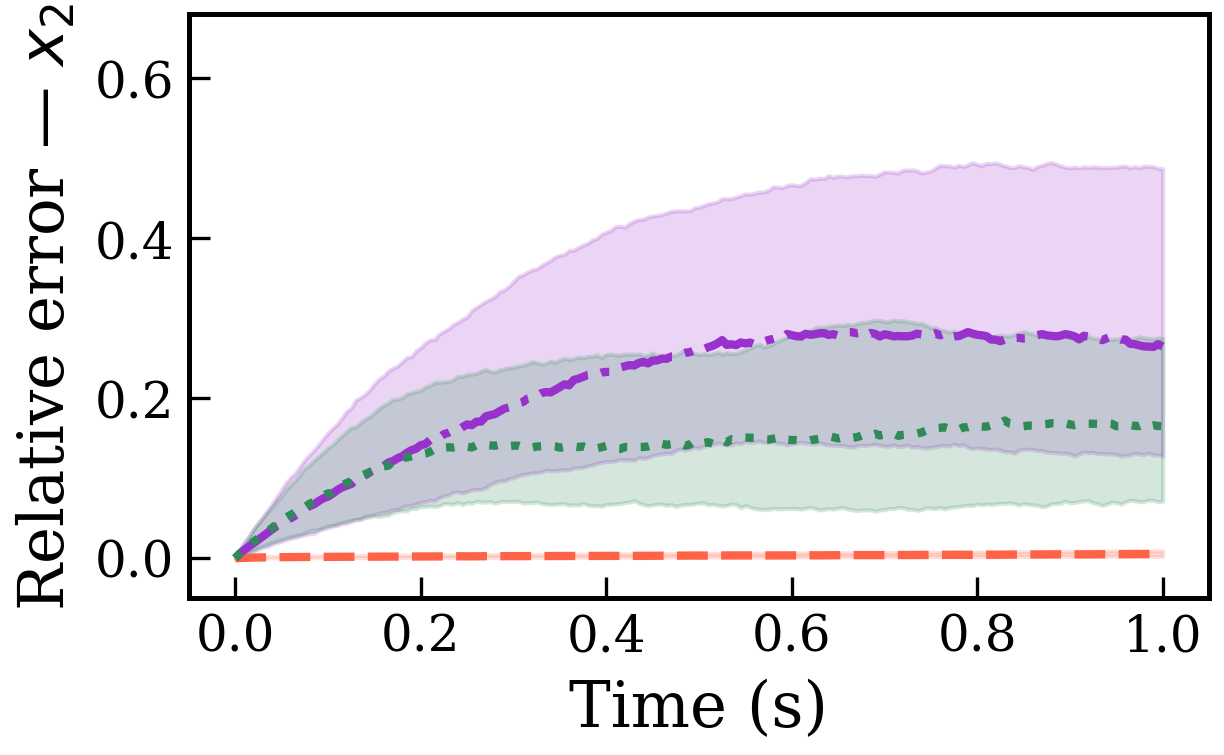}}
    \\
    {\includegraphics[width=.99\linewidth]{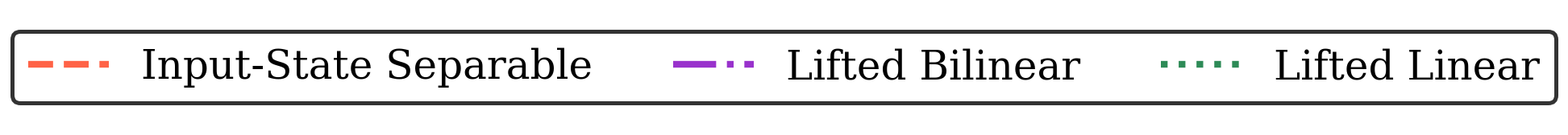}}
	\caption{Performance comparison between the input-state separable predictor and its linear/bilinear subcases. Plots show the median relative prediction error and 25th--75th percentile bounds over 1000 trajectories for states $x_1$ (left) and $x_2$ (right). Results are evaluated under two input nonlinearities: $f(u) = 2 \tanh(u)$ (top row) and $f(u) = 2 \tanh(u \cos(u))$ (bottom row).}\label{fig:DC-motor-long-term-prediction}
	\vspace*{-2ex}
\end{figure}

\section{Conclusions}
We extended the notion of consistency index to control systems, providing a computationally tractable measure of accuracy for data-driven models associated with the Koopman Control Family (KCF). We demonstrated that the control forward-backward consistency matrix possesses several desirable properties, including a real spectrum bounded within $[0,1]$ that is invariant under the choice of basis for the underlying vector spaces. Our main theoretical contribution is a sharp, closed-form upper bound on the RRMSE of KCF function predictors, determined precisely by the square root of the maximum eigenvalue of the control consistency matrix. We further showed how this bound can be specialized and applied to widely used
lifted linear and bilinear models. Finally, we demonstrated how our framework can be practically incorporated into optimization-based learning. 

\smallskip
\bibliographystyle{IEEEtran}%
\bibliography{refs}

\setcounter{section}{0} \renewcommand{\thesection}{\Alph{section}}
\section{Appendix}
\begin{lemma}\longthmtitle{Change of Non-degenerate Normal Bases}\label{l:change-normal-bases}
Given a normal space $\Sc$, let $\Psi$ and $\bar{\Psi}$ be two non-degenerate normal bases (cf.~\eqref{eq:normal-basis}) for it, and let the non-singular matrix $R \in \cplx^{\dim(\Sc) \times \dim(\Sc)}$ be the change of basis matrix satisfying $\bar{\Psi} = R \Psi$.
Given the following decompositions with appropriate dimensions
\begin{align}\label{eq:normal-basis-decomposition}
\Psi = \begin{bmatrix} I_{\Uc} \\ G \end{bmatrix} H, \quad
\bar{\Psi} = \begin{bmatrix} I_{\Uc} \\ \bar{G} \end{bmatrix} \bar{H}, \quad R= \begin{bmatrix} R_{11} & R_{12} \\ R_{21} & R_{22} \end{bmatrix},
\end{align}
we have $R_{12} = 0$, $R_{11}$ is invertible, and $\bar{H} = R_{11} H$.
\end{lemma}
 \begin{proof}
 Given the equality $\bar{\Psi} = R \Psi$ and the decompositions in~\eqref{eq:normal-basis-decomposition}, one can write
 \begin{align}\label{eq:top-block-expansion}
 I_{\Uc} \bar{H} = R_{11} I_{\Uc} H + R_{12} G H.
 \end{align}
 Now, consider an arbitrary function
 \begin{align*}
 h \in \Span(I_{\Uc} H) = \Span(I_{\Uc} \bar{H})
 \end{align*}
 with specific representations
 \begin{align}\label{eq:h-representation}
 h = v_{h}^* I_{\Uc} H = w_{h}^* I_{\Uc} \bar{H},
 \end{align}
 with $v_{h}, w_{h} \in \cplx^{\dim(\Hc)}$.
 By multiplying both sides of~\eqref{eq:top-block-expansion} from the left with $w_{h}^*$, one can write
 \begin{align*}
 w_{h}^* I_{\Uc} \bar{H} = w_{h}^* R_{11} I_{\Uc} H + w_{h}^* R_{12} G H.
 \end{align*}
 However, using~\eqref{eq:h-representation} with the equation above yields
 \begin{align*}
 (w_{h}^* R_{11} - v_{h}^*) I_{\Uc} H + w_{h}^* R_{12} G H = 0.
 \end{align*}
 The previous equation can be written in a compact form as
 \begin{align}\label{eq:compact-independence}
 \begin{bmatrix} w_{h}^* R_{11} - v_{h}^*, & w_{h}^* R_{12} \end{bmatrix} \Psi = 0.
 \end{align}
 However, since $\Psi$ is a basis, the linear independence of its elements implies $[w_{h}^* R_{11} - v_{h}^*,  w_{h}^* R_{12}] = 0$; hence
 \begin{align}\label{eq:w-orthogonality}
 w_{h}^* R_{12} = 0.
 \end{align}
 Note that~\eqref{eq:w-orthogonality} holds for all $h \in \Span(I_{\Uc} H) = \Span( I_{\Uc} \bar{H})$.
 That implies $w^* R_{12} = 0$ for all $w^* \in \cplx^{\dim(\Hc)}$.
 Hence, $R_{12} = 0$.
 The invertibility of $R_{11}$ directly follows from that $R$ is invertible and is block triangular. The rest follows from the fact that $R_{11} I_{\Uc} (u) = R_{11}$ for all $u \in \Uc$.
 \end{proof}

\end{document}